\documentclass[a4paper,12pt]{article}

\usepackage{amscd,amssymb,amsmath,amsfonts,amsthm, mathrsfs}
\usepackage{graphicx}
\usepackage{verbatim,enumerate,paralist}
\usepackage{textcomp}
\usepackage[colorlinks]{hyperref}
\hypersetup{%
  urlcolor=blue,linkcolor=blue,citecolor=blue
}
\usepackage{pdfpages}

\usepackage[T1]{fontenc}
\usepackage[all,poly,knot,tips]{xy}
\newcounter{Definitioncount}
\newtheorem{theorem}{Theorem}

\newtheorem{proposition}[theorem]{Proposition}

\theoremstyle{definition}
\newtheorem{remark}[theorem]{Remark}
\newtheorem{example}[theorem]{Example}

\newtheoremstyle{fact}{\bigskipamount}{\medskipamount}{\upshape}{}{\itshape}{. }{ }{Fact}
\theoremstyle{fact}

\newtheoremstyle{genquest}{\bigskipamount}{\medskipamount}{\upshape}{}{\itshape}{. }{ }{General Question}
\theoremstyle{genquest}

\newtheoremstyle{step}{2\bigskipamount}{\medskipamount}{\upshape}{}{\itshape}{. }{ }{\underline{Step~\thestep}}
\theoremstyle{step}

\renewcommand{\thestep}{\arabic{step}}

\newcommand{\lra}{\longrightarrow}

\newcommand{\Lra}{\Longrightarrow}

\newcommand{\ldual}[1]{\mathord{{\let\nolimits\relax\sideset{^\wedge}{}{#1}}}}
\newcommand{\laction}[2]{\mathord{{\let\nolimits\relax\sideset{^{#1}}{}{#2}}}}
\newcommand{\conj}[2]{\mathord{{\let\nolimits\relax\sideset{^{#1}}{}{#2}}}}

\def\CA{{\mathscr A}}

\def\CC{{\mathscr C}}

\def\CI{{\mathscr I}}

\def\CK{{\mathscr K}}

\def\CT{{\mathscr T}}

\def\CV{{\mathscr V}}
\def\CW{{\mathscr W}}
\def\CX{{\mathscr X}}

\begin{document}

\author{Ross Street\footnote{This author gratefully acknowledges the support of Australian Research Council Discovery Grant DP1094883.}}

\title{Skew-closed categories}
\date{1 September 2012}
\maketitle


\begin{abstract}
\noindent Spurred by the new examples found by Kornel Szlach\'anyi of a form of lax monoidal category, the author felt the time ripe to publish a reworking of Eilenberg-Kelly's original paper on closed categories appropriate to the laxer context. 
The new examples are connected with bialgebroids. 
With Stephen Lack, we have also used the concept to give an alternative definition of quantum category and quantum groupoid. Szlach\'anyi has called the lax notion {\em skew monoidal}. 
This paper defines {\em skew closed category}, proves Yoneda lemmas for categories enriched over such, and looks at closed cocompletion.  
   
\end{abstract}

\tableofcontents

\section{Introduction}\label{Intro}
\noindent While the concept of monoidal category has a longer history (for example, see \cite{Rice} and \cite{Kelly1964}), the terminology and the related concept of closed category were introduced by Eilenberg-Kelly in \cite{EilKel1966}.

Kornel Szlach\'anyi \cite{Szl2012} recently used the term skew-monoidal category for a particular laxified version of monoidal category. 
He showed that bialgebroids $H$  with base ring $R$ (a concept from \cite{Takeuchi1977} and studied under this name in \cite{Szl2003}) could be characterized in terms of skew-monoidal structures on the category of one-sided $R$-modules for which the lax unit was $R$ itself. 
More generally, it is shown in \cite{SkMon} that quantum categories \cite{QCat} with base comonoid $C$ in a suitably complete braided monoidal category $\CV$ are precisely skew-monoidal objects in $\mathrm{Comod} (\CV)$ with unit coming from the counit of $C$. 

A left skew-monoidal structure on a category $\CV$ consists of a tensor product functor $-\otimes - : \CV \otimes \CV \lra \CV$, a unit object $I$, and natural transformations $a:(A\otimes B)\otimes C \lra A\otimes (B\otimes C)$, $\ell : I\otimes A \lra A$ and $r: A \lra A\otimes I$, subject to five axioms.

In the present paper, we introduce the related concept of skew-closed category and prove  weak and strong Yoneda lemmas for categories enriched in one. 
This work was essentially done in the 1980s or earlier but remained unpublished for lack of examples. 
The author was quite surprised that a Yoneda lemma could be proved without invertibility of the associativity and unit constraints. 
The interest back then was quite pedagogical: the invertibility of the associativity constraint of the monoidal structure was at odds with the closed structure. 
In view of \cite{Szl2012} and \cite{SkMon}, the topic seems to have a renewed interest. 
Most of the ideas and calculations here are adaptations of those of \cite{EilKel1966}. 
Some involve more recent concepts.

Thanks to Ignacio Lopez Franco for noticing an error in axiom~\eqref{SCC4} in an earlier version of the paper. It was the corrected axiom that was actually used in the proofs of Example~\ref{underlying} and Proposition~\ref{closed+monoidal}.   

\section{Skew-closed categories}\label{Scc}

A {\em (left) skew-closed category} $\CV$ consists of the following data:
\begin{itemize}
\item[(i)] 
a category $\CV$,
\item[(ii)] 
a functor  $ [ -, - ] : \CV^{\mathrm{op}} \times \CV \lra \CV$,
\item[(iii)]
an object $I$ of $\CV$,
\item[(iv)]
a natural transformation $i=i_A:[I,A]\lra A$,
\item[(v)]
a natural transformation $j=j_A:I\lra [A,A]$,
\item[(vi)]
a natural transformation $L=L_{B,C}^A : [B,C] \lra [[A,B],[A,C]]$ in $\CV$,
\end{itemize}
satisfying the following five axioms.

\begin{equation}\label{SCC1}
\begin{aligned}
\xymatrix{
& [[A,C],[A,D]] \ar[rd]^-{ L}  & \\
[C,D] \ar[ru]^-{L} \ar[d]_-{L} & & [[[A,B],[A,C]],[[A,B],[A,D]]] \ar[d]^-{[L,1]} \\
[[B,C],[B,D]] \ar[rr]_-{[1,L]} & & [[B,C],[[A,B],[A,D]]] }
\end{aligned}
\end{equation}

\begin{equation}\label{SCC2}
\begin{aligned}
\xymatrix{
[[A,A],[A,C]] \ar[rr]^-{[j,1]}  && [I,[A,C]] \ar[d]^-{i} \\
[A,C] \ar[rr]_-{1} \ar[u]^-{L} && [A,C]}
\end{aligned}
\end{equation}

\begin{equation}\label{SCC3}
\begin{aligned}
\xymatrix{
[B,B] \ar[rr]^{L}   && [[A,B],[A,B]] \\
& I\ar[lu]^{j} \ar[ru]_{j}  & }
\end{aligned}
\end{equation}

\begin{equation}\label{SCC4}
\begin{aligned}
\xymatrix{
[B,C] \ar[rr]^{L} \ar[rd]_{[i,1]}   && [[I,B],[I,C]] \ar[ld]^{[1,i]}\\
& [[I,B],C]   & }
\end{aligned}
\end{equation}

\begin{equation}\label{SCC5}
\begin{aligned}
\xymatrix{
I \ar[rd]_{1}\ar[rr]^{j}   && [I,I] \ar[ld]^{i} \\
& I  &
}\end{aligned}
\end{equation}

Call $\CV$ {\em left normal} when the composite function
$$\CV (A,B) \stackrel{[A,-]}\lra \CV ([A,A],[A,B]) \stackrel{\CV (j,1)}\lra \CV (I,[A,B])$$
is invertible.

\begin{example} If $\CV$ is a left skew-monoidal category for which each functor $-\otimes B : \CV \lra \CV$ has a right adjoint $[B,-] : \CV \lra \CV$ (that is, $\CV$ is left closed) then $\CV$ becomes a left skew-closed category.  The converse is also true and we shall look into that more fully in Section~\ref{TensHom}. This is of particular interest since that converse is not true in the non-skew setting of \cite{EilKel1966}. Notice that, in the case under consideration, $\CV$ is left normal if and only if the left unit constraint $\ell : I\otimes A\lra A$ is invertible.
\end{example}

Let $\CV$ and $\CW$ be left skew-closed categories. 

A functor $F:\CV \lra\CW$ is defined to be {\em closed}  when it is equipped with a morphism 
$$\psi_0 : I \lra FI $$ and a natural transformation 
$$\psi_{A,B} : F[A,B] \lra [FA,FB] $$
satisfying the following three axioms.

\begin{equation}\label{SCF1}
\begin{aligned}
\xymatrix{
[FI,FA] \ar[rr]^-{[\psi_0,1]}  && [I,FA] \ar[d]^-{i} \\
F[I,A] \ar[rr]_-{Fi} \ar[u]^-{\psi} && FA}
\end{aligned}
\end{equation}

\begin{equation}\label{SCF2}
\begin{aligned}
\xymatrix{
FI \ar[rr]^-{Fj}  && F[A,A] \ar[d]^-{\psi} \\
I \ar[rr]_-{j} \ar[u]^-{\psi_0} && [FA,FA]}
\end{aligned}
\end{equation}

\begin{equation}\label{SCF3}
\begin{aligned}
\xymatrix{
& F[[A,B],[A,C]] \ar[rd]^-{\phantom{A}\psi}  & \\
F[B,C] \ar[ru]^-{FL\phantom{A}} \ar[d]_-{\psi} & & [F[A,B],F[A,C]] \ar[d]^-{[1,\psi]} \\
[FB,FC] \ar[rd]_-{L} & & [F[A,B],[FA,FC]]  \\
& [[FA,FB],[FA,FC]]  \ar[ru]_-{[\psi,1]} }
\end{aligned}
\end{equation}

\begin{example}\label{underlying}
For any skew-closed $\CV$, there is a closed functor 
$$\mathrm{V} : \CV \lra \mathrm{Set}$$ 
defined as follows. As a functor $\mathrm{V} = \CV (I,-)$ is represented by $I$. 
The function $\psi_0 : 1\lra  \mathrm{V}I = \CV (I,I)$ takes the single element of $1$
to the identity morphism $1_I : I \lra I$. The function 
$\psi :  \mathrm{V}[A,B] \lra [ \mathrm{V}A, \mathrm{V}B]$ takes a morphism $t:I \lra [A,B]$ to the function $\CV(I,A)\lra \CV (I,B)$ taking $a:I\lra A$ to the composite 
$$I \stackrel{t}\lra [A,B] \stackrel{[a,1]}\lra [I,B] \stackrel{i}\lra B \ .$$
In this case, \eqref{SCF1} is a tautology and \eqref{SCF2} amounts to \eqref{SCC5}.  Commutativity of \eqref{SCF3} for $F=\mathrm{V}$ amounts to showing that
 $$i_C [f,1] [[a,1],1] [i_B,1] g = i_C[1,i_C] [1,[a,1]] [f,1] L^A g$$
 for $f:I\lra [A,B]$, $g:I\lra [B,C]$ and $a:I\lra A$; this uses functoriality of 
 $[-,-]$, naturality of $L$, and \eqref{SCC4}.
\end{example}

A natural transformation $\sigma : F \Lra G : \CV \lra \CW$ is said to be {\em closed} when the following two axioms hold.

\begin{equation}\label{SCNT1}
\begin{aligned}
\xymatrix{
& I \ar[ld]_{\psi_0} \ar[rd]^{\psi_0} \\
FI\ar[rr]_{\sigma_I} & & GI
}
\end{aligned}
\end{equation}

\begin{equation}\label{SCNT2}
\begin{aligned}
\xymatrix{
F[A,B] \ar[rr]^-{\psi} \ar[d]_-{\sigma_{[A,B]}} &  & [FA,FB] \ar[d]^-{[1,\sigma_B]} \\
G[A,B] \ar[r]_-{\psi} & [GA,GB] \ar[r]_-{[\sigma_A,1]} & [FA,GB] }
\end{aligned}
\end{equation}

In this way, we have defined a 2-category $\mathrm{SCC}$ of left skew-closed categories with an underlying 2-functor $\mathrm{SCC}\lra \mathrm{Cat}$.

A comonad (in the sense of \cite{FTM}) in the 2-category $\mathrm{SCC}$ is called a {\em closed comonad} on a left skew-closed category $\CV$. 
It consists of a comonad $G:\CV \lra \CV$, $\delta : G\Lra GG$, $\varepsilon :1_{\CV} \Lra G$ which are a closed functor and closed natural transformations. 

Apart from providing examples of skew-closed categories, the next result makes the connection with \cite{ChLaSt} and \cite{SkMon}. 

\begin{proposition}
Suppose $G$ is a closed comonad on a left skew-closed category $\CV$.
Then another left skew-closed structure is defined as follows: 
\begin{itemize}
\item[(i)]
the category $\CV$,
\item[(ii)] 
the hom functor  $ < -, ? > = [G-,?] : \CV^{\mathrm{op}} \times \CV \lra \CV$,
\item[(iii)]
the unit object $I$ of $\CV$,
\item[(iv)]
the natural transformation 
$$[GI,A] \stackrel{[\psi_0,1]}\lra [I,A] \stackrel{i} \lra A \ ,$$
\item[(v)]
the natural transformation 
$$I \stackrel{j_A}\lra [A,A] \stackrel{[\varepsilon_A,1]} \lra [GA,A] \ ,$$
\item[(vi)]
the natural transformation 
$$[GA,B] \stackrel{L^{GC}}\lra [[GC,GA],[GC,B]] \stackrel{[v_{C,A},1]} \lra [G[GC,A],[GC,B]]   $$
where $v_{C,A} :G[GC,A] \lra [GC,GA]$ is the composite
$$ G[GC,A] \stackrel{\psi_{GC,A}}\lra [GGC,GA] \stackrel{[\delta_C,1]} \lra [GC,GA] \ .$$ 
\end{itemize}
\end{proposition}

\section{Categories enriched in a skew-closed category}\label{Scc-en}

Let $\CV$ be a skew-closed category. A $\CV${\em -category} $\CA$ consists of the following data:
\begin{itemize}
\item[(i)] 
a set $\mathrm{ob}\CA$ whose elements are called {\em objects} of $\CA$,
\item[(ii)] 
an object $\CA (A,B)$ of $\CV$, called a {\em hom} of $\CA$, for each pair of objects $A$ and $B$ of $\CA$,
\item[(iii)]
a morphism $j=j_A:I\lra \CA (A,A)$ for each object $A$ of $\CA$,
\item[(iv)]
a morphism $L=L_{B,C}^A : \CA (B,C) \lra [\CA (A,B),\CA (A,C)]$ for each triple of objects $A$, $B$ and $C$ of $\CA$,
\end{itemize}
satisfying the following three axioms. Notice that, as a space saving measure particularly in diagrams, we shall often write $(A,B)$ or even $(AB)$ for $\CA (A,B)$. 

\begin{equation}\label{EC1}
\begin{aligned}
\xymatrix{
& [(A,C),(A,D)] \ar[rd]^-{ L}  & \\
(C,D) \ar[ru]^-{L} \ar[d]_-{L} & & [[(A,B),(A,C)],[(A,B),(A,D)]] \ar[d]^-{[L,1]} \\
[(B,C),(B,D)] \ar[rr]_-{[1,L]} & & [(B,C),[(A,B),(A,D)]] }
\end{aligned}
\end{equation}

\begin{equation}\label{EC2}
\begin{aligned}
\xymatrix{
[(A,A),(A,C)] \ar[rr]^-{[j,1]}  && [I,(A,C)] \ar[d]^-{i} \\
(AC) \ar[rr]_-{1} \ar[u]^-{L} && (AC)}
\end{aligned}
\end{equation}

\begin{equation}\label{EC3}
\begin{aligned}
\xymatrix{
(B,B) \ar[rr]^{L}   && [(A,B),(A,B)] \\
& I\ar[lu]^{j} \ar[ru]_{j}  & }
\end{aligned}
\end{equation}

\begin{example} Notice that $\CV$ itself is a $\CV$-category. Since $\CV (A,B)$ is the hom set of the category $\CV$, we will usually write $[A,B]$ for the hom in $\CV$ as a $\CV$-category. The data $j$ and $L$ are unambiguous morphisms.
\end{example}

\begin{example} There is the {\em unit} $\CV$-category $\CI$. It has only one object $0$ and hom $\CI (0,0)=I$. Also, $j_0:I\lra \CI (0,0)$ is the identity morphism $1_I$ of $I$ while $L:\CI (0,0)\lra [\CI (0,0),\CI (0,0)]$ is $j_I : I \lra [I,I]$. 
\end{example}

A $\CV${\em -functor} $T:\CA \lra \CX$ consists of a function $T: \mathrm{ob}\CA \lra \mathrm{ob}\CX$ and morphisms 
$$T_{A,B} : \CA (A,B) \lra \CX (TA,TB) $$ 
in $\CV$, called the {\em effect on homs}, satisfying the following two axioms.

\begin{equation}\label{EF1}
\begin{aligned}
\xymatrix{
& \CX (TA,TC) \ar[rd]^-{ L}  & \\
\CA (A,C) \ar[ru]^-{T} \ar[d]_-{L} & & [\CX (TB,TA),\CX (TB,TC)] \ar[d]^-{[T,1]} \\
[\CA (B,A),\CA (B,C)] \ar[rr]_-{[1,T]} & & [\CA (B,A),\CX (TB,TC)] }
\end{aligned}
\end{equation}

\begin{equation}\label{EF2}
\begin{aligned}
\xymatrix{
\CA (A,A) \ar[rr]^{T}   && \CX (TA,TA) \\
& I\ar[lu]^{j} \ar[ru]_{j}  & }
\end{aligned}
\end{equation}

A $\CV$-functor $T:\CA \lra \CV$ might also be called a {\em left $\CA$-module}.

\begin{example}
For any $\CV$-category $\CA$ and any object $K$ of $\CA$, there is the {\em representable} $\CV$-functor 
$$\CA (K,-) : \CA \lra \CV$$
taking the object $A$ to $\CA (K,A)$ and with $L^K$ as effect on homs. 
\end{example}

We write $\CV \mathrm{Cat}$ for the category of $\CV$-categories and $\CV$-functors.

Each closed functor $F:\CV \lra \CW$ determines a functor 
\begin{equation} F_{\ast} : \CV \mathrm{Cat} \lra \CW \mathrm{Cat}
\end{equation}  
defined as follows. For each $\CV$-category $\CA$, 
the $\CW$-category $F_{\ast} \CA$ has the same objects as $\CA$,
with $j_A$ equal to the composite
$$I \stackrel{\psi_0} \lra FI \stackrel{Fj}\lra F\CA (A,A) \ ,$$
and with $L_{B,C}^A$ equal to the composite
$$F\CA (B,C) \stackrel{FL}\lra F[\CA (A,B),\CA (A,C)] \stackrel{\psi}\lra [F\CA (A,B),F\CA (A,C)] \ .$$
For each $\CV$-functor $T:\CA \lra \CX$, the $\CW$-functor $F_{\ast}T:F_{\ast}\CA \lra F_{\ast}\CX$ agrees with $T$ on objects and has effect on homs $FT_{A,B} : F\CA (A,B) \lra F\CX (TA,TB)$.  

\begin{example}\label{Forgetful} 
The closed functor $\mathrm{V} : \CV \lra \mathrm{Set}$ determines a functor
$$\mathrm{V}_{\ast} : \CV \mathrm{Cat} \lra \mathrm{Cat} \ .$$ 
A morphism $f:A\lra B$ of the category $\mathrm{V}_{\ast}\CA$ is a morphism $f:I\lra \CA (A,B)$ in $\CV$. 
Notice that if $\CV$ is left normal then there is 
a canonical isomorphism of categories 
$$\mathrm{V}_{\ast}\CV \cong \CV \ .$$
\end{example}   

A $\CV${\em -natural transformation} 
$$\theta : S\Lra T : \CA \lra \CV$$
consists of morphisms 
$$\theta_A : SA \lra TA$$
satisfying the following axiom.
\begin{equation}\label{ENTV}
\begin{aligned}
\xymatrix{
\CA (A,B) \ar[rr]^-{S} \ar[d]_-{T} && [SA,SB] \ar[d]^-{[1,\theta_B]} \\
[TA,TB] \ar[rr]_-{[\theta_A,1]} && [SA,TB]}
\end{aligned}
\end{equation}

Write $\CV^\CA$ for the category whose objects are $\CV$-functors $T : \CA \lra \CV$ and whose morphisms are $\CV$-natural transformations between such.

\begin{example}\label{ContraHom}
For each morphism $f:A \lra B$ in $\CA$ (that is, morphism  $f:I\lra \CA (A,B)$), there is a $\CV$-natural transformation $\CA (f,-)$ consisting of the morphisms $\CA (f,C) : \CA (B,C) \lra \CA (A,C)$ defined as the composites
$$\CA (B,C) \stackrel{L}\lra [\CA (A,B),\CA (A,C)]\stackrel{[f,1]}\lra [I,\CA (A,C)]\stackrel{i}\lra \CA (A,C) \ .$$
We shall see in Section~\ref{Yone} why it is a $\CV$-natural transformation and that every $\CV$-natural transformation between representable $\CV$-functors is of this form.
\end{example}

\section{The external Yoneda lemma}\label{Yone}

\begin{theorem} 
Suppose $\CV$ is a left skew-closed category, $\CA$ is a $\CV$-category, $K$ is an object of $\CA$, and $T:\CA \lra \CV$ is a $\CV$-functor. Then the assignment 
$$\CV^\CA (\CA(K,-),T) \lra \CV (I,TK) \ ,$$
taking $\theta$ to $\theta_K j$, is a bijection.
\end{theorem}
\begin{proof}
Given $\xi : I \lra TK$, define $\hat{\xi}_A$ to be the composite 
$$(KA) \stackrel{T} \lra [TK,TA] \stackrel{[\xi ,1]} \lra [I,TA] \stackrel{i} \lra TA \ .$$
To see that this family of morphisms is $\CV$-natural, consider the following diagram.
\begin{equation}\label{YonePf1}
\begin{aligned}
\xymatrix{
(AB) \ar[rrrr]^-{L} \ar[d]_-{T} &  & & & [(KA),(KB)] \ar[d]^-{[1,T]} \\
[TA,TB] \ar[dd]_-{[i,1]} \ar[rd]_-{L}  \ar[rr]^-{L} & & X \ar[rr]^-{[T,1]} \ar[rd]_-{[1,[\xi ,1]]} & & [(KA),[TK,TB]] \ar[dd]^-{[1,[\xi,1]]}  \\
& Y \ar[ld]^-{[1,i]} \ar[rr]_-{[[\xi ,1],1]} & & Z  \ar[ldld]^-{[1,i]} \\
[[I,TA],TB] \ar[rd]_{[[\xi ,1],1]} & & & & [(KA),[I,TB]] \ar[d]^-{[1,i]} \\
& [[TK,TA],TB]  \ar[rrr]_-{[T,1]} & & & [(KA),TB] }
\end{aligned}
\end{equation}
where we have put 
$$X=[[TK,TA],[TK,TB]] \ , \ Y=[[I,TA],[I,TB]] \ , \ Z=[[TK,TA],[I,TB]] \ .$$ 
In \eqref{YonePf1}, we have commutativity of the pentagon by \eqref{EF1}, of the quadrilateral involving $X$ by naturality of $L$, of the triangle involving $Y$ by \eqref{SCC5}, and of the other two regions by functoriality of $[-,-]$. Commutativity of the boundary expresses the required $\CV$-naturality; see \eqref{ENTV}.   

To see that the assignment taking $\xi$ to $\hat{\xi}$ is right inverse to the assignment in the statement of the theorem, notice that \eqref{YonePf2} commutes using \eqref{EF2}, naturality of $j$ and $i$, and \eqref{SCC5}.
 
\begin{equation}\label{YonePf2}
\begin{aligned}
\xymatrix{
I \ar[dd]_{1} \ar[rr]^{j} \ar[rd]_{j} \ar[rrd]^{j} & & (KK) \ar[d]^{T} \\
& [I,I] \ar[ld]_{i} \ar[rd]_{[1,\xi ]}& [TK,TK] \ar[d]^{[\xi , 1]} \\
I \ar[r]_{\xi} & TK & [I,TK] \ar[l]^{i} }
\end{aligned}
\end{equation}

Finally, we need to see that $\theta_A = \hat{\xi}_A$ if $\xi = \theta_K j$. For this we see that $$\hat{\xi}_A = i_{TA} [j_K ,1] [\theta_K , 1] T = i_{TA} [j_K ,1] [1, \theta_K] L$$
by \eqref{ENTV}. But, by functoriality of [-,-] and naturality of $i$, this equals $\theta_A i_{(KA)} [j_K ,1] L$, which reduces to $\theta_A$ using \eqref{SCC2}. 
\end{proof}

\section{Enriched presheaves and modules}\label{EnPresh}

Let $\CV$ be a left skew-closed category and let $\CA$ be a $\CV$-category.

A {\em presheaf} $P$ on $\CA$ (or {\em right $\CA$-module} $P$) consists of a function $P:\mathrm{ob}\CA \lra \mathrm{ob}\CV$ and morphisms 
$$P_{A,B} : PA \lra [\CA (B,A),PB]$$ 
in $\CV$ satisfying the following two axioms.  

\begin{equation}\label{EP1}
\begin{aligned}
\xymatrix{
& [(BA),PB] \ar[rd]^-{ L}  & \\
PA \ar[ru]^-{P} \ar[d]_-{P} & & [[(BC),(BA)],[(BC),PB]] \ar[d]^-{[L,1]} \\
[(CA),PC] \ar[rr]_-{[1,P]} & & [(CA),[(BC),PB]] }
\end{aligned}
\end{equation}

\begin{equation}\label{EP2}
\begin{aligned}
\xymatrix{
[(AA),PA] \ar[rr]^-{[j,1]}  && [I,PA] \ar[d]^-{i} \\
PA \ar[rr]_-{1} \ar[u]^-{P} && PA}
\end{aligned}
\end{equation}

\begin{example}\label{contahom}
For any object $K$ of $\CA$, we have the presheaf $\mathrm{y}K = \CA (-,K)$ on $\CA$
defined by $(\mathrm{y}K) A = \CA (A,K)$ and 
$(\mathrm{y}K)_{A,B} = L_{A,K}^B$.  
\end{example}

\begin{example}\label{contrahomwithfunct}
More generally, for any $\CV$-functor $T:\CA \lra \CX$ and any object $K$ of $\CX$, we have the presheaf $\CX (T-,K)$ on $\CA$
defined by $\CX (T-,K) A = \CX (TA,K)$ and 
$\CX (T-,K)_{A,B} = L_{TA,K}^{TB}$.  
\end{example}
 
Suppose $P$ and $Q$ are presheaves on $\CA$. 
A {\em presheaf morphism} $\theta : P \lra Q$ is a family of morphisms 
$$\theta_A : PA \lra QA$$
such that the following diagram commutes.

\begin{equation}\label{PM}
\begin{aligned}
\xymatrix{
PA \ar[rr]^-{P} \ar[d]_-{\theta_A} && [\CA (B,A),PB] \ar[d]^-{[1,\theta_A]} \\
QA \ar[rr]_-{Q} && [\CA (B,A),QB]}
\end{aligned}
\end{equation}

\begin{example}\label{CovarHom}
Suppose $\CV$ is left normal. For each morphism $f:A \lra B$ in $\CA$, there is a presheaf morphism 
$$\mathrm{y}f  = \CA (-,f) : \mathrm{y}A \lra \mathrm{y}B$$ 
consisting of the morphisms $\CA (C,f) : \CA (C,A) \lra \CA (C,B)$ corresponding under the bijection of left normality to the composites
$$I\stackrel{f}\lra \CA (A,B) \stackrel{L}\lra [\CA (C,A),\CA (C,B)] \ .$$
\end{example}

We also define the {\em presheaf hom} $\hat{\CA} (P,Q)$ to be the object of $\CV$ obtained as the following limit, when it exists.

\begin{equation}\label{PMO}
\begin{aligned}
\xymatrix{
\hat{\CA} (P,Q) \ar@{.>} [rr]^-{p_B} \ar@{.>}[d]_-{p_A} &  & [PB,QB] \ar[d]^-{L} \\
[PA,QA] \ar[r]_-{[1,Q]} & [PA,[(BA),QB]]  & [[(BA),PB],[(BA),QB]] \ar[l]^-{[P,1]}}
\end{aligned}
\end{equation}

There is a morphism 
$$j_P : I \lra \hat{\CA} (P,P)$$ 
induced on the limit by the morphisms 
$j_{PA }: I \lra [PA,PA]$. 
This uses \eqref{SCC3} and naturality of $j$.  

There is a morphism 
$$\mathrm{y}_{A,B} : \CA (A,B) \lra \hat{\CA} (\mathrm{y}A,\mathrm{y}B)$$  
induced on the limit by the morphisms $L_{A,B}^C$. 
We shall see in Section~\ref{StrongYone} that each $\mathrm{y}_{A,B}$ is invertible. 

\begin{remark}\label{enrichedpresheafcat} 
If all the limits $\hat{\CA} (P,Q)$ exist in $\CV$ and are preserved by each
of the functors $[X,-] : \CV \lra \CV$ then we obtain a $\CV$-category $\hat{\CA}$ whose objects are the presheaves on $\CA$ and whose homs are the objects $\hat{\CA} (P,Q)$.      
The morphism $L_{P,Q}^N$ is defined by the following commutative diagram.
\begin{equation*}
\begin{aligned}
\xymatrix{
\hat{\CA} (P,Q) \ar@{.>} [rr]^-{L_{P,Q}^N} \ar[d]_-{p_A} &  & [\hat{\CA} (N,P),\hat{\CA} (N,Q)] \ar[d]^-{[1,p_A]} \\
[PA,QA] \ar[r]_-{L^{NA}} & [[NA,PA],[NA,QA]] \ar[r]_-{[p_A,1]}  & [\hat{\CA} (N,P),[NA,QA]]}
\end{aligned}
\end{equation*}
The morphism $j_P : I \lra \hat{\CA} (N,P)$ is defined by $p_A j_P = j_{PA}$. Also we then have the {\em Yoneda} $\CV$-functor $\mathrm{y}:\CA \lra \hat{\CA}$.
\end{remark}

For $\CV$-categories $\CA$ and $\CX$, a {\em module} $\Phi : \CA \lra \CX$ consists of a family of objects $\Phi (X,A)$ of $\CV$ indexed by $X$ in $\CX$ and $A$ in $\CA$, the structure 
$$L_{\ell} : \CA (A,B) \lra [\Phi (X,A),\Phi (X,B)]$$
of $\CV$-functor $\Phi (X,-): \CA \lra \CV$ 
with object function $A \mapsto \Phi(X,A)$, and the structure 
$$L_{r} : \Phi (X,A) \lra [\CX (Y,X),\Phi (Y,A)]$$
of presheaf on $\CX$ with object function $X \mapsto \Phi(X,A)$, 
subject to the extra axiom \eqref{SCCMod}.   

\begin{equation}\label{SCCMod}
\begin{aligned}
\xymatrix{
& [\Phi (YA) \Phi (YB)] \ar[rd]^-{ L}  & \\
(AB) \ar[ru]^-{L_{\ell}} \ar[d]_-{L_{\ell}} & & [[(YX) \Phi (YA)] [(YX) \Phi (YB)]]\ar[d]^-{[L_r,1]} \\
[\Phi (XA) \Phi (XB)] \ar[rr]_-{[1,L_r]} & & [\Phi(XA) [(YX) \Phi (YB)]] }
\end{aligned}
\end{equation}

\begin{remark}\label{closedmodulewarning}
A module $\Phi : \CI \lra \CA$ is precisely a presheaf (or right module) on $\CA$. 
However, we have a word of warning. A module $\Phi : \CA \lra \CI$ amounts to a $\CV$-functor $T:\CA \lra \CV$ such that each $i:[I,FA]\lra FA$ has a right inverse subject to a condition. If $i:[I,X]\lra X$ is invertible in $\CV$ for all objects $X$ then the notions agree.   
\end{remark}

\section{A strong Yoneda lemma}\label{StrongYone}

\begin{theorem} 
Suppose $\CV$ is a left skew-closed category, $\CA$ is a $\CV$-category, $K$ is an object of $\CA$, and $P$ is a presheaf on $\CA$. Then diagram \eqref{EP1} induces an isomorphism
$$PK \cong \hat{\CA} (\mathrm{y}K,P) \ .$$
\end{theorem}
\begin{proof}
Diagram \eqref{EP1} shows that we have a cone $P_{K,B}:PK\lra [(BK),PB]$ over the appropriate diagram \eqref{PMO} with limit 
$\hat{\CA} (\mathrm{y}K,P)$. Take any cone $\xi_A : X \lra [(AK),PA]$ over that diagram. We seek a morphism $f:X\lra PK$ such that $\xi_A = P_{K,A} f$ for all $A$. By \eqref{EP2}, it follows that 
$$f=i_{PK}[j_K,1]P_{K,K} f = i_{PK}[j_K,1] \xi_K \ ,$$ 
showing the uniqueness of such an $f$. Taking $f=i_{PK}[j_K,1] \xi_K$, it remains to show $\xi_A = P_{K,A} f$ for all $A$. By \eqref{SCC2}, we have 
$$\xi_A = i_{[(AK),PA]} [j_{(AK)},1] L^A \xi_A \ .$$ 
This allows us to use \eqref{SCC3} to obtain
 $$\xi_A = i_{[(AK),PA]} [j_{(KK)},1] [L^A,1]L^A \xi_A \ ,$$ 
 and so use the cone property of $\xi$ to deduce
$$\xi_A = i_{[(AK),PA]} [j_{(KK)},1] [1,P_{K,A}] \xi_K \ .$$ 
Functoriality of $[-,-]$ and naturality of $i$ show that
$$\xi_A = i_{[(AK),PA]} [1,P_{K,A}] [j_{(KK)},1] \xi_K = P_{K,A} i_{PK} [j_{(KK)},1] \xi_K = P_{K,A} f \ ,$$
which completes the proof.      
\end{proof}

\section{Weighted colimits}\label{Wghtdcolims}

Recall from Example~\ref{contrahomwithfunct} that a $\CV$-functor $T:\CA \lra \CX$ and an object $X$ of $\CX$ determine a presheaf $\CX (T-,X)$ on $\CA$. If $J$ is another presheaf on $\CA$, we call a presheaf morphism $\kappa : J\lra \CX (T-,K)$ a {\em $J$-weighted cocone} for $T$ with {\em vertex} $K$. The composites
\begin{equation}\label{cocone}
\begin{aligned}
\xymatrix{
\CX (K,X) \ar[r]^-{L^{TA}} & [\CX (TA,K),\CX (TA,X)] \ar[r]^-{[\kappa_A,1]} & [JA,\CX (TA,X)] }
\end{aligned}
\end{equation}
for $A$ in $\CA$ induce a morphism 
\begin{equation} 
\overline{\kappa}_X : \CX (K,X) \lra \hat{\CA} (J,\CX (T-,X))
\end{equation} 
when the codomain exists. 

We say that {\em $\kappa$ exhibits $K$ as a $J$-weighted colimit of} $T$ when each 
$\overline{\kappa}_X$ is invertible; or, more precisely, when the 
morphisms \eqref{cocone} form a limit of the diagram defining 
$\hat{\CA} (J,\CX (T-,X))$. The notation is $K=\mathrm{colim}(J,T)$ or $K=J\star T$.  

We can reformulate the strong Yoneda lemma in this notation.

\begin{proposition}
If $\CV$ is skew-closed, $T:\CA \lra \CX$ is a $\CV$-functor, 
and $K$ is an object of $\CA$ then 
$$\mathrm{colim}(\mathrm{y}K,T) \cong TK \  .$$
\end{proposition}
\begin{proof} We have a presheaf morphism $\tau : \mathrm{y}K \lra \CX (T-,TK)$
whose component at $A$ is $T_{A,K}: \CA (A,K) \lra \CX (TA,TK)$. 
Theorem~\ref{StrongYone} gives an isomorphism
$$\CX (TK,X) \cong \hat{\CA}(\mathrm{y}K,\CX (T-,X))$$
which is exactly $\overline{\tau}_X$. The definition of weighted colimit gives the result.
\end{proof}

\begin{remark}\label{weightedcoliminV}
If $\CV$ is skew-closed, if each $[X,-]$ has a left adjoint $-\otimes X$, and if each $[-,X]$ takes colimits to limits, then the $J$-weighted colimit of a $\CV$-functor $T:\CA \lra \CV$ can be calculated as the ordinary colimit as in the following diagram in $\CV$.  
\begin{equation*}
\begin{aligned}
\xymatrix{
JB\otimes ((AB)\otimes TA) \ar[d]_-{1\otimes (T\otimes 1)} & (JB\otimes (AB))\otimes TA \ar[l]_-{a} \ar[r]^-{J\otimes 1} & JA \otimes TA\ar@{.>}[d]^-{q_A}  \\
JB\otimes ([TA,TB]\otimes TA) \ar[r]_-{1\otimes e} & JB\otimes TB  \ar@{.>}[r]_-{q_B} &  \mathrm{colim}(J,T)}
\end{aligned}
\end{equation*}
\end{remark}

Faced with $\CV$-modules $\Phi : \CA \lra \CX$ and $\Psi : \CX \lra \CK$, we have
presheaves $\Phi (-,A)$ on $\CX$ and $\CV$-functors $\Psi (K,-) : \CX \lra \CV$ 
and so can define 
\begin{equation}\label{modulecomposition}
(\Psi \circ \Phi)(K,A) = \mathrm{colim}(\Phi (-,A), \Psi (K,-)) \ .
\end{equation}  
Under the assumptions of Remark~\ref{weightedcoliminV} on $\CV$ plus some cocompleteness, we can make this the object assignment of a {\em composite} module 
$\Psi \circ \Phi : \CA \lra \CK$.  

\section{Monoidal skew-closed equals \\
closed skew-monoidal}\label{TensHom}

Suppose $\CV$ is a category equipped with functors 
$$ [ -, - ] : \CV^{\mathrm{op}} \times \CV \lra \CV  \ \ \mathrm{and} \ \ -\otimes - : \CV \times \CV \lra \CV \ ,$$ 
and a natural isomorphism
$$\pi : \CV (A\otimes B,C) \cong \CV(A,[B,C]) \ .$$
Let $e= \pi^{-1} (1_{[B,C]}): [B,C]\otimes B \lra C$ and $d = \pi (1_{A\otimes B}) : A\lra [B,A\otimes B] $.

By adjointness and the ordinary Yoneda lemma \cite{CWM}, the diagrams \eqref{AssocConstr1} to \eqref{AssocConstr5} below establish bijections amongst natural transformations
$$a : (A\otimes B)\otimes C \lra A\otimes (B\otimes C) \ ,$$
natural transformations
$$p:[B\otimes C,D] \lra [B,[C,D]] \ ,$$
natural transformations
$$L:[A,C] \lra [[B,A],[B,C]] \ ,$$
and natural transformations
$$M:[A,C]\otimes [B,A] \lra [B,C] \ .$$

\begin{equation}\label{AssocConstr1}
\begin{aligned}
\xymatrix{
\CV (A\otimes (B\otimes C),D) \ar[rr]^-{\pi} \ar[d]_-{\CV (a,1)} && \CV (A, [B\otimes C,D]) \ar[d]^-{\CV (1,p)} \\
\CV ((A\otimes B)\otimes C,D) \ar[rr]_-{\pi \pi} && \CV (A,[B,[C,D]])}
\end{aligned}
\end{equation}

\begin{equation}\label{AssocConstr2}
\begin{aligned}
\xymatrix{
[A,C] \ar[rd]_{[e,1]}\ar[rr]^{L}   && [[B,A],[B,C]]  \\
& [[B,A]\otimes B,C] \ar[ru]_{p}  &
}
\end{aligned}
\end{equation}

\begin{equation}\label{AssocConstr3}
\begin{aligned}
\xymatrix{
[A\otimes B,C] \ar[rd]_{L}\ar[rr]^{p}   && [A,[B,C]]  \\
& [[B,A\otimes B],[B,C]] \ar[ru]_{[d,1]}  &
}
\end{aligned}
\end{equation}

\begin{equation}\label{AssocConstr4}
\begin{aligned}
\xymatrix{
[A,C]\otimes [B,A] \ar[rd]_{L\otimes 1}\ar[rr]^{M}   && [B,C]  \\
& [[B,A],[B,C]]\otimes [B,A] \ar[ru]_{e}  &
}
\end{aligned}
\end{equation}

\begin{equation}\label{AssocConstr5}
\begin{aligned}
\xymatrix{
[A,C] \ar[rd]_{d}\ar[rr]^{L}   && [[B,A],[B,C]]  \\
& [[B,A],[A,C]\otimes [B,A]] \ar[ru]_{[1,M]}  &
}
\end{aligned}
\end{equation}

By adjointness, the diagrams \eqref{LtUnitConstr1} and \eqref{LtUnitConstr2} below establish bijections between natural transformations
$$\ell :I\otimes A \lra A$$
and natural transformations
$$j:I \lra [A,A] \ .$$

\begin{equation}\label{LtUnitConstr1}
\begin{aligned}
\xymatrix{
I \ar[rd]_{d}\ar[rr]^{j}   && [A,A]  \\
& [A,I\otimes A] \ar[ru]_{[1,\ell ]}  &
}
\end{aligned}
\end{equation}

\begin{equation}\label{LtUnitConstr2}
\begin{aligned}
\xymatrix{
I\otimes A \ar[rd]_{j\otimes 1}\ar[rr]^{\ell}   && A  \\
& [A,A]\otimes A \ar[ru]_{e}  &
}
\end{aligned}
\end{equation}

By the ordinary Yoneda lemma, the diagram \eqref{RtUnitConstr} below establishes a bijection between natural transformations
$$r :A \lra A\otimes I $$
and natural transformations
$$i:[I,B] \lra B \ .$$

\begin{equation}\label{RtUnitConstr}
\begin{aligned}
\xymatrix{
\CV (A\otimes I,B) \ar[rd]_{\cong}\ar[rr]^{\CV (r,1)}   && \CV (A,B)  \\
& \CV (A,[I,B]) \ar[ru]_{\CV (1,i)}  &
}
\end{aligned}
\end{equation}

\begin{proposition}\label{closed+monoidal} 
Suppose $\CV$ is a category and $ [ -, - ] : \CV^{\mathrm{op}} \times \CV \lra \CV$ is a functor for which each functor $[B,-]$ has a left adjoint $-\otimes B$. 
The above bijections establish a bijection between left skew-closed structures 
$$([-,-],I,i,j,L))$$ on $\CV$ and left skew-monoidal structures $$(-\otimes -,I,r, \ell , a)$$ on $\CV$.
\end{proposition}
\begin{proof} The five axioms for a left skew-closed category transform, one by one, under the bijections to the five axioms for a left skew-monoidal category.
\end{proof}

\begin{remark}
Similar considerations apply to closed functors and closed natural transformations versus monoidal functors and monoidal natural transformations in the monoidal skew-closed (= closed skew-monoidal) context.
\end{remark}

\section{A 2-category of enriched categories}\label{2SCC}

In the case where $\CV$ is left normal we can define 2-cells for the category $\CV \mathrm{Cat}$ making it a 2-category. The important observation is Example~\ref{CovarHom}. 

A $\CV${\em -natural transformation} 
$$\theta : S\Lra T : \CA \lra \CX$$
consists of morphisms 
$$\theta_A : SA \lra TA$$
satisfying the following axiom.
\begin{equation}\label{ENT}
\begin{aligned}
\xymatrix{
\CA (A,B) \ar[rr]^-{S} \ar[d]_-{T} && \CX(SA,SB) \ar[d]^-{\CX(1,\theta_B)} \\
\CX(TA,TB) \ar[rr]_-{\CX(\theta_A,1)} && \CX(SA,TB)}
\end{aligned}
\end{equation}

We emphasise that the right-hand side of the square \eqref{ENT} uses left normality of $\CV$. 
In this case, $\CV \mathrm{Cat}$ becomes a 2-category with the $\CV$-natural transformations as 2-cells and the functor of Example~\ref{Forgetful} becomes a 2-functor 
$$\mathrm{V}_{\ast} : \CV \mathrm{Cat} \lra \mathrm{Cat} \ .$$ 

\section{Categories enriched in a skew-monoidal category}\label{Smc-en}

Suppose $\CC$ is a left skew-monoidal category. 
A {\em $\CC$-category} $\CA$ consists of a set $\mathrm{ob}\CA$ of {\em objects}, a family of {\em hom objects} $\CA (A,B)$, a family of morphisms $j:I \lra \CA (A,A)$, and a family of morphisms 
$$M=M_{A,B,C} : \CA (B,C)\otimes \CA (A,B) \lra \CA (A,C) \ ,$$
where $A, B, C\in \CA$, subject to the following three axioms.

\begin{equation}\label{ECM1}
\begin{aligned}
\xymatrix{
& (CD)\otimes ((BC)\otimes (AB)) \ar[rd]^-{ 1\otimes M}  & \\
((CD)\otimes (BC))\otimes (AB) \ar[ru]^-{a} \ar[d]_-{M\otimes 1} & & (CD)\otimes (AC) \ar[d]^-{M} \\
(BD)\otimes (AB) \ar[rr]_-{M} & & (AD)}
\end{aligned}
\end{equation}

\begin{equation}\label{ECM2}
\begin{aligned}
\xymatrix{
(AB)\otimes I \ar[rr]^-{1\otimes j}  && (AB)\otimes (AA) \ar[d]^-{M} \\
(AB) \ar[rr]_-{1} \ar[u]^-{r} && (AB)}
\end{aligned}
\end{equation}

\begin{equation}\label{ECM3}
\begin{aligned}
\xymatrix{
(BB)\otimes (AB) \ar[rr]^{M}   && (AB) \\
& I\otimes (AB)\ar[lu]^{j\otimes 1} \ar[ru]_{\ell}  & }
\end{aligned}
\end{equation}

A $\CC${\em -functor} $T:\CA \lra \CX$ consists of a function $T: \mathrm{ob}\CA \lra \mathrm{ob}\CX$ and morphisms 
$$T_{A,B} : \CA (A,B) \lra \CX (TA,TB) $$ 
in $\CC$, called the {\em effect on homs}, satisfying the following two axioms.

\begin{equation}\label{MEF1}
\begin{aligned}
\xymatrix{
\CA (B,C)\otimes\CA (A,B) \ar[r]^-{T\otimes T} \ar[d]_-{M} & \CX (TB,TC)\otimes \CX (TA,TB) \ar[d]^-{M} \\
\CA (A,C) \ar[r]_-{T}  & \CX (TA,TC) }
\end{aligned}
\end{equation}

\begin{equation}\label{MEF2}
\begin{aligned}
\xymatrix{
\CA (A,A) \ar[rr]^{T}   && \CX (TA,TA) \\
& I\ar[lu]^{j} \ar[ru]_{j}  & }
\end{aligned}
\end{equation}

Suppose $\CA$ and $\CX$ are $\CC$-categories for a skew-monoidal category $\CC$. A {\em module} $\Phi : \CA \lra \CX$ consists of the following data:
\begin{itemize}
\item[(i)] 
an object $\Phi (X,A)$ of $\CC$ for all objects $A$ of $\CA$ and $X$ of $\CX$,
\item[(ii)] 
a morphism $M_{\ell} : \CA (A,B)\otimes \Phi (X,A) \lra \Phi (X,B)$ for all objects $A$ and $B$ of $\CA$ and $X$ of $\CX$,
\item[(iii)]
a morphism $M_{r} : \Phi (X,A) \otimes \CX (Y,X)  \lra \Phi (Y,A)$ for all objects $A$ of $\CA$ and $X$ and $Y$ of $\CX$,
\end{itemize}
subject to the following five axioms (where we also abbreviate $\Phi (X,A)$ to $(XA)$ to save space).

\begin{equation}\label{EMM1}
\begin{aligned}
\xymatrix{
& (AB)\otimes ((XA)\otimes (YX)) \ar[rd]^-{ 1\otimes M_r}  & \\
((AB)\otimes (XA))\otimes (YX) \ar[ru]^-{a} \ar[d]_-{M_{\ell}\otimes 1} & & (AB)\otimes (YA) \ar[d]^-{M_{\ell}} \\
(XB)\otimes (YX) \ar[rr]_-{M_r} & & (YB)}
\end{aligned}
\end{equation}

\begin{equation}\label{EMM2}
\begin{aligned}
\xymatrix{
& (BC)\otimes ((AB)\otimes (XA)) \ar[rd]^-{ 1\otimes M_{\ell}}  & \\
((BC)\otimes (AB))\otimes (XA) \ar[ru]^-{a} \ar[d]_-{M\otimes 1} & & (BC)\otimes (XB) \ar[d]^-{M_{\ell}} \\
(AC)\otimes (XA) \ar[rr]_-{M_{\ell}} & & (XC)}
\end{aligned}
\end{equation}

\begin{equation}\label{EMM3}
\begin{aligned}
\xymatrix{
& (XA)\otimes ((YX)\otimes (ZY)) \ar[rd]^-{ 1\otimes M}  & \\
((XA)\otimes (YX))\otimes (ZY) \ar[ru]^-{a} \ar[d]_-{M_r\otimes 1} & & (XA)\otimes (ZX) \ar[d]^-{M_r} \\
(YA)\otimes (ZY) \ar[rr]_-{M_r} & & (ZA)}
\end{aligned}
\end{equation}

\begin{equation}\label{EMM4}
\begin{aligned}
\xymatrix{
(XA)\otimes I \ar[rr]^-{1\otimes j}  && (XA)\otimes (XX) \ar[d]^-{M_r} \\
(XA) \ar[rr]_-{1} \ar[u]^-{r} && (XA)}
\end{aligned}
\end{equation}

\begin{equation}\label{EMM5}
\begin{aligned}
\xymatrix{
(AA)\otimes (XA) \ar[rr]^{M_{\ell}}   && (XA) \\
& I\otimes (XA)\ar[lu]^{j\otimes 1} \ar[ru]_{\ell}  & }
\end{aligned}
\end{equation}

\begin{remark}\label{Transport} 
Suppose we are in the monoidal skew-closed situation of Section~\ref{TensHom}. Then the data and axioms for enriched category, enriched functor and module based on the skew-monoidal category $\CV$ transform across the tensor-hom adjunction isomorphism $\pi$ to the data and axioms for enriched category, enriched functor and module based on the skew-closed category $\CV$. 
\end{remark}

\begin{remark}\label{monoidalmodulewarning}
Let $\CC$ be skew-monoidal. We ambiguously write $\CI$ for the unit $\CC$-category (it has one object $0$, $\CI (0,0)=I$, $M=\ell_I$ and $j_0 = 1_I$).  A module $\Phi : \CI \lra \CA$ is precisely a {\em right $\CA$-module}. 
However, we have a word of warning as in Remark~\ref{closedmodulewarning}. A module $\Phi : \CA \lra \CI$ is less general than a {\em left $\CA$-module}. If $r:X\lra X\otimes I$ is invertible in $\CC$ for all objects $X$ then the notions agree.   
\end{remark}

\section{Skew-promonoidal categories and convolution}\label{SkProConv}

Promonoidal categories and convolution were introduced in Day \cite{DayPhD}; also see \cite{DayConv}. As pointed out in \cite{SkMon}, this carries over to the skew context.

A {\em (left) skew-promonoidal category} is a category $\CT$ equipped with functors 
$$P: \CT^{\mathrm{op}}\times \CT^{\mathrm{op}}\times \CT \lra \mathrm{Set} \ , \quad \ J:\CT \lra \mathrm{Set}$$
and natural transformations 
$$a:\int^XP(A,X;D)\times P(B,C;X) \lra \int^X P(X,C;D)\times P(A,B;X) \ ,$$
$$\ell : \CT (A,B) \lra \int^X JX\times P(X,A;B) \ ,$$
$$r:\int^X P(A,X;B)\times JX \lra \CT (A,B)$$
satisfying the following five axioms where we write $(ABC)$ for $P(A,B;C)$, use an Einstein-like notation $SX \times TX$ for $\int^XSX \times TX$ (so that integrating over repeated variables is understood), and even omit the symbol $\times$.
\begin{equation}\label{SPMC1}
\begin{aligned}
\xymatrix{
(AXE)(BYX)(CDY) \ar[r]^-{ a\times 1} \ar[d]_-{ 1\times a}  & (XYA)(ABX)(CDY) \ar[d]^-{\cong} \\
(AYE)(XDY)(BCX)\ar[d]_-{a \times 1} & (XYA)(CDY)(ABX) \ar[d]^-{a \times 1} \\
(YDE)(AXY)(BCX) \ar[r]_-{1\times a} &(YDE)(XCY)(ABX)}
\end{aligned}
\end{equation}

\begin{equation}\label{SPMC2}
\begin{aligned}
\xymatrix{
(AXD)(BCX)JB \ar[r]^-{ a\times 1}  & (XCD)(ABX)JB \ar[d]^-{1\times r} \\
(AXD)JB(BCX) \ar[u]^-{\cong} & (XCD)\CT (A,X) \ar[d]^-{\cong} \\
(AXD)\CT (C,X) \ar[u]^-{1\times \ell} \ar[r]_-{1\times a} & (ACD)}
\end{aligned}
\end{equation}

\begin{equation}\label{SPMC3}
\begin{aligned}
\xymatrix{
JA(AXD)(BCX) \ar[r]^-{ 1\times a}  & JA(XCD)(ABX)\ar[r]^-{ \cong} & (XCD)JA(ABX)  \\
(AXD)JB(BCX) \ar[u]^-{\cong} && (XCD)\CT (A,X) \ar[u]_-{1\times \ell}\\
\CT (X,D)(BCX) \ar[u]^-{\ell \times 1} \ar[rr]_-{\cong} & & (BCD) \ar[u]_-{\cong}}
\end{aligned}
\end{equation}

\begin{equation}\label{SPMC4}
\begin{aligned}
\xymatrix{
(AXD)(BCX)JC \ar[r]^-{ a\times 1} \ar[d]_-{ 1\times r}  & (XCD)(ABX)JC \ar[d]^-{\cong} \\
(AXD)JB(BCX) \ar[d]_-{\cong}  & (XCD)JC(ABX) \ar[d]^-{r\times 1} \\
(ABD) \ar[r]_-{\cong} & \CT (X,D)(ABX)}
\end{aligned}
\end{equation}

\begin{equation}\label{SPMC5}
\begin{aligned}
\xymatrix{
JA(AXB)JX \ar[r]^-{ 1\times r}   & JA\CT (A,B) \ar[d]^-{\cong} \\
\CT (X,B)JX \ar[r]_-{\cong} \ar[u]^-{\ell \times 1}  & JB} 
\end{aligned}
\end{equation}

The next proposition provides many examples of skew-promonoidal categories. 
We follow that with an example not covered by the proposition. 

\begin{proposition}
A skew-promonoidal category $\CT$ in which $P(-,?;C)$ and $J$ are representable is a skew-monoidal category. A skew-promonoidal category $\CT$ in which $P(A,-;?)$ and $J$ are representable is a skew-closed category.  
\end{proposition}

\begin{example}
Take any category $\CC$ and any object $Z$ in it.
Define $$P(A,B;C) = \CC (A,Z)\times \CC(B,C)$$ and $$JC = \CC (Z,-) \ .$$
Using Yoneda, we have 
$$\int^XP(A,X;D)\times P(B,C;X) \cong \CC(A,Z)\times \CC (C,D)\times \CC (B,Z)$$
and
$$\int^X P(X,C;D)\times P(A,B;X) \cong \CC(B,Z)\times \CC (C,D)\times \CC (A,Z) \ ,$$
and so an obvious switch isomorphism for the natural transformation $a$. 
Also by Yoneda,
$$\int^X JX\times P(X,A;B) \cong \CC (Z,Z)\times \CC (A,B)$$
and
$$\int^X P(A,X;B)\times JX\cong \CC (A,Z) \times \CC (Z,B) \ ,$$
and so we put $\ell (h) = (1_Z , h)$ and $r(f,g)=gf$.
This defines a Hopf skew-promonoidal structure on $\CC$ for every object $Z$. 
\end{example}

We shall give two examples of convolution on particular skew-promonoidal categories.

The left skew-monoidal cocompletion of a (small) left skew-monoidal category $\CC$ is the functor (presheaf) category 
$$\hat{\CC} = [\CC^{\mathrm{op}},\mathrm{Set}] \ . $$ 
The skew-monoidal structure on $\hat{\CC}$ is defined by the convolution tensor product formula (see \cite{CWM} for the integral notation of Yoneda for ends and coends):
\begin{equation}
\begin{aligned}
(M\ast N)C = \int^{A,B} \CC (C,A\otimes B)\times MA \times NB
\end{aligned}
\end{equation}
with unit defined by
\begin{equation}
\begin{aligned}
JC = \CC (C,I) \ .
\end{aligned}
\end{equation}

The Yoneda embedding $\mathrm{y} : \CC \lra \hat{\CC}$ is strong monoidal. 
This means the skew-monoidal structure on $\hat{\CC}$ restricts along
this fully faithful functor $\mathrm{y}$ to the given such structure on $\CC$, up to isomorphism.  

In fact, $\hat{\CC}$ is closed; the left internal hom is defined by the formula
\begin{equation}
\begin{aligned}{
[N,K]A = \int_{B,C} [\CC (C,A\otimes B)\times NB,KC] \ ,}
\end{aligned}
\end{equation} 
where the square brackets denote the set of functions. By Proposition~\ref{closed+monoidal}, $\hat{\CC}$ is left skew-closed and we can speak of a $\hat{\CC}$-category $\CA$ in the sense of Section~\ref{Scc-en}. By Remark~\ref{Transport}, such an $\CA$ is bijectively equivalent to a $\hat{\CC}$-category $\CA$ in the sense of Section~\ref{Smc-en}. Moving from $\CC$ to $\hat{\CC}$ then allows us to deduce external and strong Yoneda lemmas for categories enriched in the skew-monoidal category $\CC$. Of course, these lemmas can be proved directly without such a transcendental argument. We leave the reader to formulate those lemmas.

For the second example, suppose $\CV$ is a (small) skew-closed category. There is a convolution closed \textbf{right} skew-monoidal structure on the functor category
$$ [\CV,\mathrm{Set}] $$ 
defined by
\begin{equation}
\begin{aligned}
{(M\ast N)C = \int^{B} M[B,C] \times NB}
\end{aligned}
\end{equation}
with unit defined by
\begin{equation}
\begin{aligned}
{JC = \CV (I,C) \ .}
\end{aligned}
\end{equation}
The right internal hom is defined by the formula
\begin{equation}
\begin{aligned}
{[M,K]B = \int_{C} [M[B,C],KC]   \  .}
\end{aligned}
\end{equation}
In other words,
\begin{equation}
\begin{aligned}
{[M,K]B =  [\CV, \mathrm{Set}] (M[B,-],K-) \  .}
\end{aligned}
\end{equation}
The left unit constraint $\ell_M :M\lra J\ast M$ is defined by 
$$MC \stackrel{j_C\times 1}\lra \CV (I,[C,C])\times MC \stackrel{\mathrm{copr}_C}\lra \int^B \CV(I,[B,C])\times MB \ .$$
The right unit constraint $r_M : M\ast J = M[I,-] \lra M$ is $Mi$.
The associativity constraint $a:M\ast (N\ast K)\lra (M\ast N)\ast K$ is induced by the composite of
\begin{equation*}
\begin{aligned}
\xymatrix{
M[B,C]\times N[E,B]\times KE \ar[rr]^{ML_{B,C}^E\times 1\times 1\phantom{AAA}}  & &  M[[E,B],[E,C]]\times N[E,B]\times KE   
}
\end{aligned}
\end{equation*}
with the coprojection $\mathrm{copr}_{[E,B], E}$ into the coend 
$$((M\ast N)\ast K)C=\int^{D,E} M[D,[E,C]]\times ND\times KE \ .$$
\begin{center}
--------------------------------------------------------
\end{center}

\appendix

\end{document}